\documentclass[11pt]{article}

\usepackage[a4paper,margin=1in]{geometry}
\usepackage{amsmath,amssymb,amsthm}
\usepackage{mathtools}
\usepackage{graphicx}
\usepackage{tikz}
\usepackage{caption}
\usepackage{subcaption}
\usepackage[hidelinks]{hyperref}
\usepackage[nameinlink,capitalise,noabbrev]{cleveref} 

\title{A Permanental Analog of the Rank--Nullity Theorem for Symmetric Matrices}

\DeclareMathOperator{\per}{per}
\DeclareMathOperator{\rank}{rank}
\author{
Priyanshu Pant\\
Department of CSE, IIT Indore, India\\
\texttt{priyanshupant03@gmail.com}
\and
Surabhi Chakrabartty\\
Department of Mathematical Sciences, IISER Berhampur, India\\
\texttt{surabhic20@iiserbpr.ac.in}
\and
Ranveer Singh\\
Department of CSE, IIT Indore, India\\
\texttt{ranveer@iiti.ac.in}
}

\date{} 

\theoremstyle{plain}
\newtheorem{theorem}{Theorem}[section]   
\newtheorem{lemma}[theorem]{Lemma}

\theoremstyle{definition}

\newtheorem{example}[theorem]{Example}

\theoremstyle{remark}

\begin{document}
\maketitle
\begin{abstract}
The rank of an $n \times n$ matrix $A$ is equal to the size of its largest square submatrix with a nonzero determinant; it can be computed in $O(n^{2.37})$ time. Analogously, the size of the largest square submatrix with nonzero permanent is defined as the permanental rank $\rho_{\mathrm{per}}(A)$. Computing the permanent or coefficients of the permanental polynomial $\per(xI-A)$ is \#P-complete.
The permanental nullity  $\eta_{\mathrm{per}}(A)$ is defined as the multiplicity of zero as a root of the permanental polynomial.
We establish a permanental analog of the rank–nullity theorem, $\rho_{\mathrm{per}}(A) + \eta_{\mathrm{per}}(A) = n$ for symmetric nonnegative matrices, positive semidefinite matrices, and adjacency matrices of balanced signed graphs. Using this theorem, we can compute the permanental nullity for symmetric nonnegative matrices and adjacency matrices of balanced signed graphs in polynomial time.
For symmetric $\{0,\pm 1\}$-matrices, we also 
provide a complete characterization of when the permanental rank-nullity identity holds. 
\end{abstract}

\section{Introduction}

The \textit{determinant} of an $n \times n$ matrix $A=(a_{ij})$ is defined as 
\begin{equation}
    \det (A) = \sum_{\sigma \in S_{n}} \operatorname{sgn}(\sigma) \prod_{i=1}^{n} a_{i,\sigma(i)}, \notag
\end{equation}
where $S_n$ is the set of all permutations of $\{1, 2, \ldots, n\}$ and $\operatorname{sgn}$ is the sign of the permutation $\sigma$. The \textit{permanent} is defined in a similar way, but without the $\operatorname{sgn}$ factor:
\begin{equation}
    \per(A) = \sum_{\sigma \in S_{n}} \prod_{i=1}^{n} a_{i,\sigma(i)}. \notag
\end{equation}

Although the definition differs only by a sign factor, the computational complexity of the determinant and permanent is believed to vary significantly. The determinant can be computed in time $\mathcal{O}(n^{2.37})$ \cite{kaltofen2005complexity}, while finding the permanent is known to be $\#P$-complete \cite{valiant1979complexity}. The permanent lacks a well-established algebraic or geometric interpretation, and it is neither multiplicative nor invariant under linear combinations of rows or columns. As a result, the permanent has received less attention in the early literature, with nearly all known results at the time compiled in the book~\cite{minc1984permanents}. In 1979, Valiant \cite{valiant1979complexity} proved that computing the permanent is $\#P$-complete, thus making the permanent a central object of study in computational complexity theory.

The Pólya Permanent Problem \cite{polya1913aufgabe} investigates conditions under which the permanent of a given matrix can be transformed into the determinant of a modified matrix. This problem is equivalent to twenty-three other combinatorial and graph-theoretic problems \cite{mccuaig2004polya}, including the enumeration of perfect matchings in bipartite graphs. This motivates the search for algebraic frameworks in which the permanent satisfies identities structurally similar to those of the determinant.

Unlike the determinant, the permanent generally lacks multiplicativity, that is, 
\(
\operatorname{per}(AB) \neq \operatorname{per}(A) \operatorname{per}(B)
\).
Marcus and Minc~\cite{marcus1965permanents} conjectured that multiplicativity holds only when both matrices are products of permutation and diagonal matrices. Beasley~\cite{beasley1969maximal} later proved this, showing it is the maximal class for which the permanent is multiplicative. Beyond multiplicativity, several determinant inequalities have been adapted for the permanent. Lieb~\cite{lieb2002proofs} established a reversed version of the classical Fischer inequality for permanents. Specifically, for a block matrix
\(
A = \begin{pmatrix}
B & C \\
C^* & D
\end{pmatrix} \succeq 0,
\)
where \( A \) is positive semidefinite, the inequality 
\(
\operatorname{per}(A) \geq \operatorname{per}(B) \operatorname{per}(D)
\) holds.
This stands in contrast to the determinant case, where the inequality is reversed, that is, \( \det(A) \leq \det(B) \det(D) \) under the same conditions. Carlen, Lieb, and Loss also proved a Hadamard-type bound, relating the permanent to the product of the row norms~\cite{carlen2006inequality}.
Heuvers, Cummings, and Bhaskara Rao~\cite{heuvers1988characterization} showed that the permanent satisfies an identity analogous to the Cauchy-Binet formula for the determinant.

Such analogs between determinants and permanents motivate a question: which matrix properties defined in terms of the permanent remain tractable despite the computational hardness of the permanent?

A fundamental matrix property is its \textit{rank}. The classical $\operatorname{rank}(A)$, defined as the size of the largest square submatrix with a non-zero determinant, is efficiently computable. 
A natural analog is the \textit{permanental rank} $\rho_{\mathrm{per}}(A)$, defined as the size of the largest square submatrix with a non-zero permanent. Yang Yu~\cite{yu1999permanent} introduced this concept and established the fundamental inequality $\rank(A) \leq 2 \rho_{\per}(A)$, which is known to be tight. This work connected permanental rank to combinatorial matrix theory and inspired further research, including connections to the Alon--Jaeger--Tarsi conjecture~\cite{alon1989nowhere} and partial transversals in Latin squares~\cite{fanai2012permanent}.

Analogous to the characteristic polynomial, the \textit{permanental polynomial} of an
$n \times n$ matrix \(A\) is defined as
\begin{equation}
    \pi(A, x) = \per(xI-A), \notag
\end{equation}
where \(I\) is the identity matrix of order $n$. Given a graph $G$ with adjacency matrix $A(G)$, its permanental polynomial is $\pi(G, x) = \per(xI-A(G))$. The multiset of all roots of $\pi(G, x)$, including multiplicities, is called the \textit{per-spectrum} of $G$.
Turner \cite{turner1968generalized} first introduced permanental polynomials in graph theory, and Merris et al. \cite{merris1981permanental} and Kasum et al. \cite{kasum1981chemical} later expanded this work to explore their mathematical properties and uses in chemistry.
For more results on permanental polynomials and per-spectrum, see~\cite{bapat2024computing, cash2000permanental,gutman1998permanents, liu2013characterizing,   singh2024note,  zhang2015per}.

The \textit{permanental nullity}, \( \eta_{\text{per}}(G) \) defined as the multiplicity of zero as a root in \( \pi(G, x) \) was first studied by Wu and Zhang \cite{wu2015per}. They established its connection to the matching number via the Gallai--Edmonds structure theorem. They provided exact characterizations for graphs with extremal per-nullities \( n - 2, n - 3, n - 4, n - 5 \), and derived a sharp formula for general graphs involving factor-critical components. Their work also characterized graphs with zero permanental nullity and analyzed their behavior on unicyclic graphs, line graphs, and factor-critical graphs. Unlike the determinant case, permanents do not admit an eigenvalue–eigenvector
theory, so \(\eta_{\mathrm{per}}(A)\) is an algebraic notion and does not coincide
with the geometric nullity.

For a symmetric matrix, it is known that the rank equals the number of nonzero eigenvalues, while the nullity corresponds to the number of zero eigenvalues. In this paper, we demonstrate that a similar relationship exists for the permanental roots of a symmetric matrix. We show that the identity $\rho_{\mathrm{per}}(A) + \eta_{\mathrm{per}}(A) = n$ holds for several matrix classes, including nonnegative symmetric matrices, positive semidefinite matrices, and symmetric matrices corresponding to balanced signed graphs. We also provide a necessary and sufficient condition under which the identity $\rho_{\mathrm{per}}(A) + \eta_{\mathrm{per}}(A) = n$ holds for  $\{0,\pm1\}$-symmetric matrices. Our work demonstrates that despite the \#P-hardness of permanent computation, permanental rank and nullity exhibit surprising algorithmic tractability for some classes of matrices.

The remainder of the paper is structured as follows. Section \ref{perliminaties} introduces the necessary definitions, notations, and some preliminary results. Section \ref{mainresults} presents our main theorems, establishing the permanental rank-nullity relationship for a few classes of symmetric matrices, and discusses counterexamples where this relation fails. Finally, Section \ref{futuredirections} outlines directions for future work.

\section{Preliminaries}\label{perliminaties}

 
Signed graphs were introduced by Harary~\cite{harary1953notion} and later formalized by Zaslavsky~\cite{zaslavsky1982signed}. A signed graph $G_\sigma$ is an undirected graph where each edge $(u, v)$ has a sign $\sigma(u, v)$ which is either positive (+1) or negative (-1). The corresponding \emph{signed adjacency matrix} \( A_\sigma \in \{0,\pm 1\}^{n\times n}
 \) is defined by
\[
A_\sigma(u,v) = 
\begin{cases}
\sigma(u,v) & \text{if } (u,v) \in E(G), \\
0 & \text{otherwise}.
\end{cases}
\]

In a signed graph, a cycle is called \emph{positive} if it contains an even number of negative edges, and \emph{negative} if it contains an odd number of negative edges. A signed graph $G_\sigma$ is called \emph{balanced} if every cycle is positive. The following result provides a simple criterion to determine whether a signed graph is balanced.

\begin{lemma} [\cite{zaslavsky2013matrices}] \label{prop:harary-matrix}
A signed graph \(G_\sigma\) is balanced if and only if its signed adjacency matrix \(A_\sigma\) is diagonally similar to the adjacency matrix \(A\) of the corresponding unsigned graph, that is,
\(
A_\sigma = D A D,
\)
for some diagonal matrix \(D = \mathrm{diag}(\pm1,\ldots,\pm1)\).
\end{lemma}

For a \( n\times n \) matrix \( A \) and index sets \( I,J \subseteq [n] \), we use \( A[I , J] \) to denote the submatrix of \( A \) formed by rows indexed by \( I \) and columns indexed by \( J \).  
Let \( S \subseteq [n] \) be an index set of size \( k \). We write \( A[S,S] \) to denote the \( k \times k \) principal submatrix of \( A \) formed by selecting the rows and columns indexed by \( S \).
We recall a few classical results on positive semidefinite matrices related to their permanents.

\begin{lemma}[\cite{Marcus1963ThePA}]\label{theorem:psd-properties}

Let \(A \in \mathbb{R}^{n \times n}\) be a positive semidefinite symmetric matrix.
Then:
\begin{enumerate}
  \item \(\per(A) \ge 0\).
  \item \(a_{ii}a_{jj} \ge a_{ij}^2\).
  \item \(\per(A) \ge \prod_{i=1}^n a_{ii}\).
  \item Every principal submatrix of \(A\) is positive semidefinite; hence
  \(\per(A[I,I]) \ge 0\) for all \(I \subseteq [n]\).
\end{enumerate}
\end{lemma}

We now state some results about the coefficients of the permanental polynomial of a matrix.

\begin{lemma}[\cite{minc1984permanents}]\label{lemma:coeff-principal} 
Let \(A \in \mathbb{R}^{n \times n}\) be a matrix
with permanental polynomial
\[
\pi(A, x) =\per(xI-A)= \sum_{i=0}^n b_i x^{n-i}.
\]
Then \( b_0 = 1 \), and for \( 1 \leq i \leq n \),
\[
b_i = (-1)^i \sum_{|S| = i} \operatorname{per}(A[S, S]),
\]
where the sum is over all principal \( i \times i \) submatrices of \( A \).
\end{lemma}

Interestingly, for the adjacency matrix of a signed graph $G_{\sigma}$, the coefficients $b_i$ can be expressed in terms of Sachs subgraphs. A \emph{Sachs subgraph} is a subgraph in which every connected component is either
an edge or a cycle (a loop counts as a $1$-cycle when diagonal entries are allowed). For signed graphs, the signed adjacency matrix has zero diagonal, so loops do not arise. The coefficients of the permanental polynomial are therefore given by sums over Sachs subgraphs, as shown below.

\begin{lemma}[\cite{tang2022permanental}]\label{lemma:sachs-signed}
Let \( A_\sigma \) be the signed adjacency matrix of a signed graph \( G_\sigma \) on \( n \) vertices, and let
\[
\pi(A_\sigma,x) = \sum_{i=0}^n s_i x^{n-i}
\]
be its permanental polynomial. Then for \( 0 \leq i \leq n \),
\[
s_i = (-1)^i \sum_{U_i} (-1)^{c^-(U_i)} 2^{c(U_i)},
\]
where the sum is over all signed Sachs subgraphs \( U_i \) of \( G_\sigma \) on \( i \) vertices, \( c(U_i) \) is the number of cycles in \( U_i \), and \( c^-(U_i) \) is the number of negative cycles in \( U_i \).
\end{lemma}

\section{Main Results} \label{mainresults}

In this section, we first establish a general inequality for the permanental rank and permanental nullity that holds for all square matrices. We then provide a necessary and sufficient condition under which the permanental rank–nullity identity holds for $\{0,\pm1\}$-symmetric matrices. Using this condition, we show that the permanental rank–nullity identity holds for nonnegative symmetric matrices and adjacency matrices of balanced signed graphs. We also prove that the identity holds for positive semidefinite matrices using a different argument based on their structure.

\subsection{General Inequality}

\begin{theorem}\label{prop:general-inequality}
For any square matrix \( A \in \mathbb{R}^{n \times n} \),  
\[
\rho_{\operatorname{per}}(A) + \eta_{\operatorname{per}}(A) \geq n.
\]
\end{theorem}

\begin{proof}
Let the permanental polynomial of \( A \) be
\[
\pi(A,x) = \sum_{i=0}^n b_i x^{n - i}.
\]
By Lemma~\ref{lemma:coeff-principal}, for each \( 1 \leq i \leq n \),
\[
b_i = (-1)^i \sum_{|S| = i} \operatorname{per}(A[S, S]),
\]
where the sum is over all principal \( i \times i \) submatrices of \( A \).
Assume that  \(\rho_{\per}(A) =k \). Then every principal submatrix of order greater than \(k\) has zero permanent. So, for all $i > k$, the coefficients
\[
b_i = (-1)^i \sum_{|S| = i} \operatorname{per}(A[S,S]) = 0.
\]
The permanental polynomial simplifies to
\[
\begin{aligned}
\pi(A, x) &= \sum_{i = 0}^{k} b_i x^{n - i} \\
          &=  b_0 x^n + b_1 x^{n-1} + \dots + b_k x^{n - k} \\
          &= x^{n - k}(b_0 x^k + b_1 x^{k - 1} + \dots + b_k).
\end{aligned}
\]
Since \( x^{n-k} \) is a factor, the polynomial must have at least \( n - k \) roots equal to zero. So, by the definition of permanental nullity, we have
\(
\eta_{\per}(A) \geq n - k.
\)

Therefore, 
\(
\eta_{\per}(A) + \rho_{\per}(A) \geq n.
\)\end{proof}

 While Theorem~\ref{prop:general-inequality} holds universally, computing either \(\rho_{\operatorname{per}}(A)\) or \(\eta_{\operatorname{per}}(A)\) directly appears challenging due to the \#P-hardness of permanent computation. However, as we will show, for some matrix classes, both parameters become tractable, and the inequality becomes an equality.
The inequality in Theorem~\ref{prop:general-inequality} can be strict without further assumptions, as shown in the example below.
\begin{example} \label{example_gen}
Let 
\(
A = \begin{pmatrix}
  0 &  1 \\
  0 &  0
\end{pmatrix}.
\)
Then,
\(
\rho_{\per}(A) = 1 \), \( \eta_{\per}(A) = 2 \), so
\(
\rho_{\per}(A) + \eta_{\per}(A) = 3 > n = 2.
\)

\end{example}

Given a matrix \(A=[a_{ij}]_{n\times n}\), let \(\vec G\) denote the directed graph
on vertex set \([n]\) with an arc \(i\to j\) whenever \(a_{ij}\neq 0\).
A \emph{directed cycle cover} of a vertex set \(S\subseteq[n]\) is a vertex-disjoint union of directed cycles covering all vertices of \(S\).
For any directed cycle cover \(C\), define its weight by
\[
w(C):=\prod_{(i\to j)\in C} a_{ij}.
\]

The permanent of \(A\) can be written as~(see \cite{BrualdiCvetkovic2008}),
\begin{equation}\label{eq:per-cycle-covers}
\per(A)=\sum_{C\in\mathcal L} w(C),
\end{equation}
where \(\mathcal L\) denotes the set of directed cycle covers of \(\vec G\).
Combining Lemma~\ref{lemma:coeff-principal} with~\eqref{eq:per-cycle-covers}, for
\(1\le i\le n\) we obtain
\begin{equation}\label{eq:bi-cycle-covers}
b_i
= (-1)^i \sum_{\substack{S\subseteq[n]\\ |S|=i}}
\ \sum_{C\in\mathcal L(S)} w(C),
\end{equation}
where \(\mathcal L(S)\) denotes the set of directed cycle covers of the subgraph of
\(\vec G\) induced on the vertex set \(S\).
When \(A=A_\sigma\) is the signed adjacency matrix of a signed graph \(G_\sigma\),
we have \(b_i=s_i\), and by Lemma~\ref{lemma:sachs-signed},
\begin{equation}\label{eq:bi-sachs}
b_i = (-1)^i \sum_{U_i} (-1)^{c^-(U_i)} 2^{c(U_i)},
\end{equation}
where the sum ranges over all signed Sachs subgraphs \(U_i\) of \(G_\sigma\) on \(i\)
vertices.

\subsection{A General Criterion for Symmetric \texorpdfstring{$\{0,\pm1\}$}{\{0,±1\}}-Matrices}

Let \(A=[a_{ij}]\in\{0,\pm1\}^{n\times n}\) be symmetric, and let \(\vec G\) be the
directed graph associated with \(A\) as defined above.
For \(S\subseteq[n]\), let \(\mathcal L(S)\) denote the set of directed cycle covers
of the subgraph of \(\vec G\) induced on \(S\), allowing loops \(i\to i\) whenever
\(a_{ii}\neq 0\).
For every \(C\in\mathcal L(S)\), we have \(w(C)\in\{\pm1\}\).
Define
\[
\begin{aligned}
E_i &:= (-1)^i \sum_{\substack{S\subseteq[n]\\ |S|=i}}
\bigl|\{C\in\mathcal L(S): w(C)=+1\}\bigr|,\\
O_i &:= (-1)^i \sum_{\substack{S\subseteq[n]\\ |S|=i}}
\bigl|\{C\in\mathcal L(S): w(C)=-1\}\bigr|.
\end{aligned}
\]

By~\eqref{eq:bi-cycle-covers}, for every \(i\),
\[
b_i = E_i - O_i.
\]

When \(A\) is the signed adjacency matrix of a signed graph, then by Lemma~\ref{lemma:sachs-signed} this identity admits an equivalent form with
\[
E_i = (-1)^{i} \sum_{\substack{U_i\ \\ c^-(U_i)\ \text{is even}}} 2^{c(U_i)},
\qquad
O_i = (-1)^{i} \sum_{\substack{U_i\ \\ c^-(U_i)\ \text{is odd}}} 2^{c(U_i)}.
\]

\begin{theorem}\label{thm:general-case}
Let \(A\in\{0,\pm1\}^{n\times n}\) be a symmetric matrix and let
\(\rho_{\operatorname{per}}(A)=k\). Then
\[
\rho_{\operatorname{per}}(A)+\eta_{\operatorname{per}}(A)=n
\quad\text{if and only if}\quad
E_k\neq O_k.
\]
\end{theorem}

\begin{proof}
Since \(\rho_{\operatorname{per}}(A)=k\), we have \(b_i=0\) for all \(i>k\).
Hence
\[
\pi(A,x)
=\sum_{i=0}^k b_i x^{n-i}
=x^{n-k}\bigl(b_0x^k+b_1x^{k-1}+\cdots+b_k\bigr).
\]
The permanental nullity \(\eta_{\operatorname{per}}(A)\) is the multiplicity of
\(0\) as a root of \(\pi(A,x)\).
Therefore,
\(
\eta_{\operatorname{per}}(A)=n-k
\quad\Longleftrightarrow\quad
b_k\neq 0
\quad\Longleftrightarrow\quad
E_k\neq O_k.
\)
\end{proof}
Theorem~\ref{thm:general-case} reduces the verification of the permanental rank-nullity identity to checking whether \(E_k \neq O_k\). While counting directed cycle covers and Sachs subgraphs is generally computationally hard, this characterization enables efficient algorithms for restricted graph classes where such enumeration is tractable.

\subsection{Nonnegative Symmetric Matrices}
We begin by giving an alternate graph-theoretic definition of the permanental rank,
which applies to all nonnegative matrices.

Let \( A \) be a nonnegative \( n \times n \) matrix.
Since every term in the expansion of \(\per(A)\) is nonnegative, replacing each
nonzero entry of \(A\) with \(1\) does not affect the permanental rank.

\begin{lemma}\label{lemma:perm-rank}
Let \(A\) be a nonnegative \(n\times n\) matrix.
Then \(\rho_{\operatorname{per}}(A)=k\) if and only if \(k\) is the maximum number of
arcs in a directed subgraph of \(\vec G\) in which each vertex has indegree and
outdegree at most \(1\).
\end{lemma}

\begin{proof}
Let there exist index sets \(I,J\subseteq[n]\) with \(|I|=|J|=k\) such that
\(\per(A[I,J])\neq 0\).
Hence there is a bijection \(\pi:I\to J\) with \(a_{i,\pi(i)}\neq 0\) for all
\(i\in I\).
Let
\(
F=\{(i,\pi(i)) : i\in I\}.
\)
Then \(F\) is a set of \(k\) arcs in \(\vec G\) such that, in the subgraph of
\(\vec G\) induced by the arcs in \(F\), each vertex has indegree and outdegree at most \(1\).

Conversely, suppose \(F\) is a set of \(k\) arcs in \(\vec G\) such that, in the
subgraph of \(\vec G\) induced by the arcs in \(F\), each vertex has indegree and outdegree at
most \(1\).
Let \(I\) be the set of vertices with outdegree \(1\), and \(J\) the set of
vertices with indegree \(1\).
Then \(|I|=|J|=k\), since \(F\) has \(k\) arcs, and \(F\) defines a bijection \(\pi:I\to J\) with
\(a_{i,\pi(i)}\neq 0\).
Hence \(\per(A[I,J])\neq 0\), so \(\rho_{\operatorname{per}}(A)\ge k\).

The first part shows that any \(k\times k\) submatrix with nonzero permanent
gives a set of \(k\) arcs in \(\vec G\) in which each vertex has indegree
and outdegree at most \(1\), while the second part shows that any set of \(k\)
arcs in \(\vec G\) with this property implies there exists a \(k\times k\) submatrix with
nonzero permanent. Hence \(\rho_{\operatorname{per}}(A)\) is equal to the maximum
size of a subgraph of \(\vec G\) in which every vertex has indegree and outdegree
at most \(1\).
\end{proof}

We now consider nonnegative symmetric matrices.

\begin{lemma}\label{lemma:principal-existence}
Let \(A\) be a nonnegative symmetric matrix of order \(n\) with
\(\rho_{\operatorname{per}}(A)=k\). Then there exists a principal submatrix
\(A[S,S]\) of order \(k\) such that \(\per(A[S,S])\neq 0\).
\end{lemma}

\begin{proof}
Since \(\rho_{\per}(A)=k\), there exist index sets \(I,J\subseteq[n]\) with
\(|I|=|J|=k\) such that \(\per(A[I,J])\neq 0\).
By Lemma~\ref{lemma:perm-rank}, there exists a set \(F\) of \(k\) arcs in \(\vec G\)
such that, in the subgraph induced by the arcs in \(F\), each vertex of \(I\) has outdegree
exactly \(1\), each vertex of \(J\) has indegree exactly \(1\), vertices in
\(I\setminus J\) have indegree \(0\), and vertices in \(J\setminus I\) have
outdegree \(0\).

If $I=J$, the statement holds. Now, assume $I\neq J$. We will now prove that if an arc \(u\to v\in F\) with \(v\notin I\), then $u \in J$. Suppose for contradiction \(u\notin J\). Then \(u\in I\setminus J\) and \(v\in J\setminus I\).
Since the arc \(u\to v\in F\), we have \(a_{u,v}\neq 0\), and by symmetry
\(a_{v,u}\neq 0\). Consider the set of arcs
\(
F' := F \cup \{v\to u\}.
\)
Because \(u\) has indegree \(0\) in \(F\) and \(v\) has outdegree \(0\) in \(F\),
the set \(F'\) satisfies the condition that every vertex has indegree and
outdegree at most \(1\). But \(F'\) consists of \(k+1\) arcs, which contradicts
the assumption \(\rho_{\per}(A)=k\) by Lemma~\ref{lemma:perm-rank}. Therefore
\(u\in J\).

We now construct a directed cycle cover on the vertex set \(I\) from \(F\).
Let arc \(u\to v\in F\). If \(v\in I\), retain the arc \(u\to v\). If \(v\notin I\), then by the claim \(u\in J\) there exists a unique vertex
\(z\in I\) such that the arc \(z\to u\in F\), since \(u\in J\) has indegree exactly \(1\). By symmetry of \(A\), the arc \(u\to z\) exists.
Replace the arc \(u\to v\) by the arc \(u\to z\).
Moreover, no vertex of \(I\) receives more than one incoming arc; otherwise one can modify \(F\) to obtain a set of \(k+1\) arcs with indegree and outdegree at most \(1\), contradicting Lemma~\ref{lemma:perm-rank}. Doing this for all arcs of \(F\), we get \(k\) arcs on \(k\) vertices
of \(I\), and every vertex of \(I\) has outdegree \(1\).
It follows that each vertex of \(I\) has indegree \(1\) as well, and hence we
obtain a directed cycle cover on \(I\).

Thus there exists a permutation \(\sigma\) of \(I\) such that \(a_{i,\sigma(i)}\neq 0\) for all \(i\in I\).
Since \(A\) is nonnegative, this gives \(\per(A[I,I])\neq 0\), completing the proof.
\end{proof}

Given a nonnegative symmetric matrix \(A\), let \(G\) denote the associated
undirected graph on vertex set \([n]\), where \(\{i,j\}\in E(G)\) if
\(a_{ij}\neq 0\), with a loop at \(i\) whenever \(a_{ii}\neq 0\).
The following lemma reformulates the existence of a directed cycle cover in terms
of Sachs subgraphs of \(G\).

\begin{lemma}\label{perrank_sachsub}
Let \(A\) be a nonnegative symmetric matrix, and let \(G\) be the undirected graph
associated with \(A\) as above. If \(\rho_{\operatorname{per}}(A)=k\), then \(G\)
contains a Sachs subgraph on \(k\) vertices.
\end{lemma}

\begin{proof}
Let \( \rho_{\operatorname{per}}(A)=k \).
By Lemma~\ref{lemma:principal-existence}, there exists an index set
\( S\subseteq[n] \) with \( |S|=k \) such that
\(\per(A[S,S])\neq 0\).
Hence there exists a permutation \(\sigma\) of \(S\) such that
\(a_{i,\sigma(i)}\neq 0\) for all \(i\in S\).
The subgraph of \(G\) induced by the edge set \(\{\{i,\sigma(i)\} : i\in S\}\)
(including a loop at \(i\) when \(\sigma(i)=i\)) is a vertex-disjoint union of
cycles, and hence a Sachs subgraph on \(k\) vertices.
\end{proof}

\begin{theorem}\label{thm:general-prn}
Let \( A \) be a nonnegative symmetric matrix. Then
\[
\rho_{\operatorname{per}}(A) + \eta_{\operatorname{per}}(A) = n.
\]
\end{theorem}

\begin{proof}
Let \(\rho_{\operatorname{per}}(A)=k\).
By Lemma~\ref{lemma:principal-existence}, there exists \(S\subseteq[n]\) with
\(|S|=k\) such that \(\per(A[S,S])\neq 0\).
Write
\(\pi(A,x)=\sum_{i=0}^n b_i x^{n-i}\).
By Lemma~\ref{lemma:coeff-principal},
\[
b_k = (-1)^k \sum_{\substack{S\subseteq[n]\\ |S|=k}} \per(A[S,S]).
\]
Since \(A\) is nonnegative, all terms in the sum are nonnegative and at least one
is positive, so \(b_k\neq 0\).
By definition of \(\rho_{\operatorname{per}}(A)=k\), we have \(b_i=0\) for all
\(i>k\).
Thus
\[
\pi(A,x)=x^{n-k}(b_0x^k+\cdots+b_k),
\]
with \(b_k\neq 0\), and hence \(0\) is a root of multiplicity \(n-k\).
Therefore \(\eta_{\operatorname{per}}(A)=n-k\), proving the claim.
\end{proof}

An alternative proof of Theorem~\ref{thm:general-prn} is given in Appendix~\ref{appendix}.

Lemma~\ref{lemma:perm-rank} gives a polynomial-time procedure for computing
\(\rho_{\per}(A)\).
Construct a bipartite graph with left vertices \(i^{-}\) and right vertices
\(i^{+}\) for \(i\in[n]\), and add an edge from \(i^{-}\) to \(j^{+}\) whenever
\(A_{ij}\neq 0\).
Then \(\rho_{\per}(A)\) equals the size of a maximum matching in this graph, which
can be computed in polynomial time~\cite{hopcroft1973n}.
By Theorem~\ref{thm:general-prn}, \(\eta_{\per}(A)=n-\rho_{\per}(A)\), and hence
\(\eta_{\per}(A)\) is also computable in polynomial time.

The example below shows that the identity 
\(
\rho_{\mathrm{per}}(A) + \eta_{\mathrm{per}}(A) = n
\)
does not always hold in general for all symmetric matrices with both positive and negative entries. 

\begin{example}
Consider the symmetric matrix
\(
B = \begin{pmatrix}
0  & 0  & 1  & -1 \\
0  & 0  & 1  & 1  \\
1  & 1  & 0  & 1  \\
-1 & 1  & 1  & 0
\end{pmatrix}
\). We find that \(
\per(B) = 0
\)
and consider the \(3 \times 3\) principal submatrix corresponding to indices \(\{2,3,4\}\), that is,
\(
B' = \begin{pmatrix}
0 & 1 & 1 \\
1 & 0 & 1 \\
1 & 1 & 0
\end{pmatrix}.
\) Then,
\(
\per(B') = 2 \neq 0.
\)
Thus,
\(
\rho_{\per}(B) = 3.
\)
For the permanental nullity, we obtain the permanental polynomial,
\(
\pi(B,x) = x^4 + 5x^2 = x^2(x^2 + 5).
\)
The root \(x = 0\) has multiplicity 2, hence
\(
\eta_{\per}(B) = 2.
\)
Therefore,
\(
\rho_{\per}(B) + \eta_{\per}(B) = 3 + 2 = 5 \neq 4 = n.
\)

\end{example}

\subsection{Adjacency Matrices of Balanced Signed Graphs}

In Section~3.2, we established that the permanental rank–nullity identity holds for $\{0,\pm1\}$-symmetric matrix \(A\) if and only if \(E_k \neq O_k\), where \(k = \rho_{\mathrm{per}}(A)\). We now show that this condition is always satisfied when \(A\) corresponds to a balanced signed graph.

 Let \( A \in \{0,1\}^{n \times n} \) be the nonnegative symmetric matrix obtained by replacing all \(-1\) entries in \( A_\sigma \) with \( 1 \). Then, the graph \( G \) with adjacency matrix \( A\) is the underlying unsigned graph of $G_\sigma$. Note that \(A\) is an adjacency matrix, so \(a_{ii}=0\) for all \(i\), and hence the associated graph \(G\) has no loops.

\begin{theorem}\label{thm:balanced-prn}
Let \( A_\sigma \in \{0, \pm 1\}^{n \times n} \) be a symmetric matrix such that the associated signed graph \( G_\sigma \) is balanced. Then  
\[
\rho_{\operatorname{per}}(A_\sigma) + \eta_{\operatorname{per}}(A_\sigma) = n.
\]
\end{theorem}

\begin{proof}

By Lemma~\ref{prop:harary-matrix}, there exists a diagonal matrix \( D = \operatorname{diag}(\pm 1, \ldots, \pm 1) \) such that
\(
A_\sigma = D A D,
\)
where \( A \) is the adjacency matrix of the underlying unsigned graph \( G \), which is symmetric and has nonnegative entries.
Now for any index sets \( I, J \subseteq [n] \) with \( |I| = |J| \), the corresponding submatrix satisfies
\[
A_\sigma[I, J] = D[I, I] \times A[I, J] \times D[J, J],
\]
where \( D[I, I] \) and \( D[J, J] \) are diagonal matrices with entries \( \pm 1 \).

Then, the permanent satisfies
\[
\begin{aligned}
\operatorname{per}(A_\sigma[I, J]) 
&= \operatorname{per}(D[I, I] \times A[I, J] \times D[J, J]) \\
&= \left( \prod_{i \in I} D_{ii} \right) 
   \left( \prod_{j \in J} D_{jj} \right) 
   \operatorname{per}(A[I, J]).
\end{aligned}
\]

Hence, $\operatorname{per}(A_\sigma[I, J]) \ne 0$  if and only if $\operatorname{per}(A[I, J]) \ne 0$. 
Therefore,
\(
\rho_{\mathrm{per}}(A_\sigma) = \rho_{\mathrm{per}}(A).
\)

Now, let \( \rho_{\operatorname{per}}(A_\sigma) = k \), then \( \rho_{\operatorname{per}}(A) = k \). By Lemma~\ref{perrank_sachsub}, the underlying graph \( G \) contains a Sachs subgraph on \( k \) vertices, and hence \(E_k \neq 0\).
Since \(G_\sigma\) is balanced, every cycle is positive; hence \(c^-(U_k)=0\) for every
Sachs subgraph \(U_k\). Therefore \(O_k=0\).

Therefore, \( E_k \neq O_k \), and by Theorem~\ref{thm:general-case}, we conclude
\(
\rho_{\operatorname{per}}(A_\sigma) + \eta_{\operatorname{per}}(A_\sigma) = n.
\)
\end{proof}




Since it can be tested in linear time whether a given signed graph is balanced (see \cite{HARARY1980131}), and \(\rho_{\per}(A)\) is computable via matching, both \(\rho_{\per}(A)\) and \(\eta_{\per}(A)\) are polynomial-time computable for balanced signed graphs.

\subsection{Positive Semidefinite Matrices}

 This class of matrices has a well-known structure, and we use some known properties of their permanents to carry out the proof.
 
\begin{lemma}\label{lemma:psd-principal-submatrix}
Let \( A \in \mathbb{R}^{n \times n} \) be positive semidefinite, and suppose
\(\rho_{\operatorname{per}}(A) = k\).
Then any \(k\times k\) submatrix \(A[I,J]\) of \(A\) with
\(\operatorname{per}(A[I,J]) \neq 0\) must be a principal submatrix, that is,
\(I = J\).
\end{lemma}

\begin{proof}
By Theorem~\ref{prop:general-inequality} for any square matrix \( A \) of order $n$,
\[
\rho_{\operatorname{per}}(A) + \eta_{\operatorname{per}}(A) \ge n.
\]
Let \( \rho_{\operatorname{per}}(A) = k \), and let \( A[I, J] \) be any
$k\times k$ submatrix of \( A \) such that
\(
\operatorname{per}(A[I, J]) \neq 0.
\) Then for each  \(i \in I\) there exists a  \(j \in J \)  such that  \(a_{ij} \neq 0\).

By property (2) of Lemma~\ref{theorem:psd-properties}, for such \( i \) and \( j \), we have \( a_{ii} \neq 0 \) and \( a_{jj} \neq 0 \). Hence, the submatrix \( A[I \cup J, I \cup J] \) has all diagonal entries nonzero. Then, by property (3) of Lemma \ref{theorem:psd-properties},
\(
\operatorname{per}(A[I \cup J, I \cup J]) \neq 0.
\)

If \( I \neq J \), then \( |I \cup J| > k \), contradicting  \( \rho_{\operatorname{per}}(A) = k \). Therefore, we must have \( I = J \), and hence any $k\times k$ submatrix of \(A\)
with non-zero permanent is necessarily a principal submatrix.
This completes the proof.
\end{proof}

\begin{theorem}\label{thm:psd-prn}
Let \( A \in \mathbb{R}^{n \times n} \) be positive semidefinite.  Then
\[
\rho_{\operatorname{per}}(A) \;+\; \eta_{\operatorname{per}}(A) \;=\; n.
\]
\end{theorem}

\begin{proof}
By Lemma \ref{lemma:psd-principal-submatrix}, there exists a principal submatrix $A[S,S]$ of order $k = \rho_{\operatorname{per}}(A)$ with $\operatorname{per}(A[S,S]) \neq 0$. So the permanental polynomial of $A$ is given by, 
\(\pi(A,x) = x^{n-k} \left( b_0 x^k + \cdots + b_{k} \right), \) with \( b_{k} \neq 0\) by property (4) of Lemma \ref{theorem:psd-properties}. Thus, the multiplicity of $0$ as a root is exactly $n - k$, implying  $\eta_{\operatorname{per}}(A) = n - k$.

Hence, $
\eta_{\per}(A)+\rho_{\per}(A)=n.
$ \end{proof}

This result extends naturally to positive semidefinite Hermitian matrices,
since the arguments in Lemmas~\ref{lemma:psd-principal-submatrix}
and Theorem~\ref{thm:psd-prn} rely only on matrix symmetry, which is preserved under
conjugate transposition.

Unlike previous classes, no polynomial-time algorithm is known for computing \(\rho_{\operatorname{per}}(A)\) or \( \eta_{\operatorname{per}}(A)\) for general PSD matrices.

\section{Future Directions}\label{futuredirections}
In this paper, we proved that the permanental rank–nullity identity
\[
\rho_{\mathrm{per}}(A) + \eta_{\mathrm{per}}(A) = n
\]
holds for nonnegative symmetric matrices, positive semidefinite matrices, and adjacency matrices of balanced signed graphs. More generally, we showed that the identity holds for any $\{0,\pm1\}$-symmetric matrix \(A\) if and only if \(E_k \neq O_k\), where \(k = \rho_{\mathrm{per}}(A)\). 
We can extend the condition to any real matrix by letting each entry \(A_{ij}\) be the weight on the directed edge \(i\to j\), and then comparing the total weights of positive and negative weighted cycles in directed cycle covers on \(k\) vertices.
 Although this condition extends to all real matrices, it only becomes more meaningful if we can relate the condition
\(
E_k \neq O_k
\)
to matrix properties, such as rank, eigenvalue spectrum, permanental spectrum, or other structural patterns.
Can the permanental rank and nullity be used to derive structural results about matrices or graphs, and do they interact with other invariants in interesting ways? Developing such results may lead to a broader theory parallel to linear algebra.

On the algorithmic side, because computing the permanent is \#P-complete, it is important to ask what is the computational complexity of computing \(\rho_{\mathrm{per}}(A)\) and  \(\eta_{\mathrm{per}}(A)\) for arbitrary real symmetric matrices? Is it NP-hard? For positive semidefinite matrices, we established the rank-nullity identity but did not provide an efficient algorithm. Is \(\rho_{\mathrm{per}}(A)\) and  \(\eta_{\mathrm{per}}(A)\) polynomial-time computable for PSD matrices? If not, what is its approximation complexity? 
Since Theorem~\ref{thm:general-case} involves counting directed cycle covers,
can the problem of deciding whether \(E_k \neq O_k\) be placed within the
polynomial hierarchy?

We believe these directions will lead to a more complete theory of the permanent that complements its well-studied algebraic properties, bridging matrix theory, computational complexity, and graph algorithms.

\appendix
\section{Alternative proof of Theorem~\ref{thm:general-prn}} \label{appendix}

\begin{lemma}\label{lemma:app-nonempty-intersection}
Let $A$ be a non-negative symmetric matrix with $\rho_{\operatorname{per}}(A) = k$. For any $k \times k$ submatrix $A[I,J]$ with $\operatorname{per}(A[I,J]) \neq 0$, we have $I \cap J \neq \emptyset$.
\end{lemma}
\begin{proof}
Let us assume that \( I \cap J = \emptyset \) and that the submatrix \( A[I, J] \) satisfies
\[
\per(A[I, J]) \neq 0.
\]
Because \( A \) is symmetric, we can define another submatrix.
\[
M' = \begin{bmatrix} X & M \\ M^T & Y \end{bmatrix},
\]
where 
\(
M = A[I, J], X = A[I, I], Y = A[J, J].
\)

 Since \( A \) is non-negative, every entry in \( A \) (and thus in \( M, X, Y \)) is non-negative. Moreover, the symmetry of \( A \) implies that \( M^T = A[J, I] \). Thus \(
\per(M') \geq \per(M) \, \per(M^T)
\) (see \cite{minc1984permanents}).
Also,
\(\per(M) = \per(M^T)\) \& \( \per(M) \neq 0. \) Hence,
\(\per(M') \neq 0.\) However, \( M' \) is a submatrix of \( A \) of order \( |I| + |J| = 2k > k \), which contradicts the assumption that \( \rho_{\per}(A) = k \). Thus, the assumption \( I \cap J = \emptyset \) must be false, implying
\(
I \cap J \neq \emptyset.
\) 
\end{proof}

\begin{lemma}\label{lemma:app-bijection-constraint}
  Let \( A \) be a non-negative symmetric matrix with \( \rho_{\per}(A) = k \). 
  Then, any permutation $\sigma$ corresponding to which a submatrix $A[I, J]$ of order $k$ has non-zero permanent, must satisfy the following condition:
    \[ \sigma(i) \notin I \implies i \in J\]
\end{lemma}

\begin{proof}
Since \( \per(A[I, J]) \neq 0 \) and \( A \) is non-negative, there exists a bijection 
$
\sigma: I \to J
$
such that
$
A(i, \sigma(i)) \neq 0 \quad \text{for all } i \in I.
$

By Lemma \ref{lemma:app-nonempty-intersection} we know \( I \cap J \neq \phi \). Now, suppose there exists some \( i \in I \) such that
$
\sigma(i) \notin I \quad \text{and} \quad i \notin J.
$
 Since \( A(i, \sigma(i)) \neq 0 \) and $A$ is symmetric, \( A(\sigma(i),i) \neq 0 \). Then the extended submatrix
$
A[I \cup \{\sigma(i)\}, J \cup \{i\}],
$
which is of order \( k+1 \) will also have nonzero permanent, contradicting the definition of \( \rho_{\per}(A) = k \).
Thus, we must have
$
\sigma(i) \notin I \implies i \in  J.$
\end{proof}

\begin{lemma}\label{lemma:app-principal-existence}
Let \( A \) be a non-negative symmetric matrix of order \( n \) with \( \rho_{\per}(A) = k \). Then, there exists a principal submatrix \( A[S, S] \) of order \( k \) such that
\[
\per(A[S, S]) \neq 0.
\]
\end{lemma}

\begin{proof}

Since $\rho_{\per}(A)=k$, so there exists a submatrix \( A[I, J] \) (with \( I, J \subseteq \{1, 2, \dots, n\} \) and \( |I| = |J| = k \)) such that
\(
\per(A[I, J]) \neq 0.
\)

If $I=J$, then the proof is complete. So, we will consider the case when $I \neq J$. Now, as $\per(A[I, J]) \neq 0$, by definition of permanent there exists a bijection $\sigma \in S_n$ such that for each $i \in I$, there exists a $ \sigma(i) \in J$ and 
$ A(i,\sigma(i)) \neq 0. $

Now, we show that the principal submatrix \( A[I, I] \) has a nonzero permanent. Define a permutation \( \pi: I\to I \) as follows.
\[
\pi(i) = 
\begin{cases} 
\sigma(i), & \text{if } \sigma(i) \in I, \\[1mm]
\sigma^{-1}(i), & \text{if } \sigma(i) \notin I.
\end{cases}
\]
Since, by Lemma \ref{lemma:app-bijection-constraint} if \(\sigma(i) \notin I \implies i \in J \), \( \pi(i) \) is well-defined for all \( i \in I\). \\
Now we verify two claims about \( \pi \):

\textbf{Claim 1: }  \( \pi \) is a bijection on \( I \).

We consider three cases for \( i_1, i_2 \in I \):

\begin{enumerate}
    \item  If \( \sigma(i_1), \sigma(i_2) \in I \), then \( \pi(i_1) = \sigma(i_1) \) and \( \pi(i_2) = \sigma(i_2) \). Since \( \sigma \) is injective, \( \pi(i_1) = \pi(i_2) \) implies \( i_1 = i_2 \).
    
    \item  If \( \sigma(i_1), \sigma(i_2) \notin I \), then \( \pi(i_1) = \sigma^{-1}(i_1) \) and \( \pi(i_2) = \sigma^{-1}(i_2) \). Since \( \sigma^{-1} \) is injective, again \( \pi(i_1) = \pi(i_2) \) implies \( i_1 = i_2 \).
    
    \item If \( \sigma(i_1) \in I \) but \( \sigma(i_2) \notin I \), then \( \pi(i_1) = \sigma(i_1) \) and \( \pi(i_2) = \sigma^{-1}(i_2) \). Suppose \( \pi(i_1) = \pi(i_2) = c \). Then \( \sigma(i_1) = c \) and \( \sigma^{-1}(i_2) = c \), so \( \sigma(c) = i_2 \). By the defining property of \( \sigma \), we have \( A(i_1, \sigma(i_1)) = A(i_1, c) \neq 0 \). Since \( A \) is symmetric, \( A(i_2, c) = A(c, i_2) = A(i_2, \sigma(i_1)) \neq 0 \). This contradicts the bijectivity of \( \sigma \) unless \( i_1 = i_2 \).
\end{enumerate}
Thus, \( \pi \) is injective (and hence bijective on \( I \)).

\textbf{Claim 2:} For each \( i \in I \), \( A(i,\pi(i)) \neq 0 .\)  

There are two cases to consider:
\begin{enumerate}
    \item If \( \sigma(i) \in I \) then \( \pi(i) = \sigma(i) \), and by definition of \( \sigma \), we have
    \(
    A(i, \pi(i)) = A(i, \sigma(i)) \neq 0.
    \)
    
    \item If \( \sigma(i) \notin I \), then by Lemma \ref{lemma:app-bijection-constraint}, it follows that \( i \in J \). Thus, there exists a unique \( i' \in I \) such that \( \sigma(i') = i \) that is, \( i' = \sigma^{-1}(i) \). By the definition of \( \sigma \),
    \(
    A(i', \sigma(i)) = A(i', i) \neq 0.
    \)
    Since \( A \) is symmetric, we also have \( A(i, i') = A(i, \sigma^{-1}(i)) \neq 0 \). In this case, \( \pi(i) = \sigma^{-1}(i) \), so once again \( A(i, \pi(i)) \neq 0 \).
\end{enumerate}

Thus, for every \( i \in I \), \( A(i,\pi(i))\neq 0 \). It follows that
\(
\prod_{i \in I} A(i,\pi(i)) \neq 0,
\)
and hence,
\(
\per(A[I, I]) \neq 0.
\)
This shows that \( A[I,I] \) is a principal submatrix of order \( k \) with nonzero permanent, which completes the proof. \\
\end{proof}

\begin{proof}[Proof of Theorem~\ref{thm:general-prn}]
Assume \(\rho_{\per}(A)=k\). By Lemma~\ref{lemma:app-principal-existence}, there exists
\(S\subseteq[n]\) with \(|S|=k\) such that \(\per(A[S,S])\neq 0\).
Write \(\pi(A,x)=\sum_{i=0}^n b_i x^{n-i}\).
By Lemma~\ref{lemma:coeff-principal},
\[
b_k = (-1)^k\sum_{|T|=k}\per(A[T,T])\neq 0.
\]
Since \(\rho_{\per}(A)=k\), we have \(b_i=0\) for all \(i>k\), hence
\[
\pi(A,x)=x^{n-k}(b_0x^k+\cdots+b_k),
\]
with \(b_k\neq 0\). Therefore \(0\) has multiplicity \(n-k\), i.e.,
\(\eta_{\per}(A)=n-k\), and \(\rho_{\per}(A)+\eta_{\per}(A)=n\).
\end{proof}

\bibliographystyle{plainurl}
\bibliography{references}

@article{valiant1979complexity,
  title={The complexity of computing the permanent},
  author={Valiant, Leslie G},
  journal={Theoretical computer science},
  volume={8},
  number={2},
  pages={189--201},
  year={1979},
  publisher={Elsevier},
  url = {https://doi.org/10.1016/0304-3975(79)90044-6}
}

@article{kaltofen2005complexity,
  author  = {Kaltofen, Erich and Villard, Gilles},
  title   = {On the complexity of computing determinants},
  journal = {Computational Complexity},
  volume  = {13},
  number  = {3},
  pages   = {91--130},
  year    = {2005},
  issn    = {1420-8954},
  url     = {https://doi.org/10.1007/s00037-004-0185-3},
}

@article{turner1968generalized,
  title={Generalized matrix functions and the graph isomorphism problem},
  author={Turner, James},
  journal={SIAM Journal on Applied Mathematics},
  volume={16},
  number={3},
  pages={520--526},
  year={1968},
  publisher={SIAM},
  url={https://doi.org/10.1137/0116041}
}

@article{merris1981permanental,
  title={Permanental polynomials of graphs},
  author={Merris, Russell and Rebman, Kenneth R and Watkins, William},
  journal={Linear Algebra and Its Applications},
  volume={38},
  pages={273--288},
  year={1981},
  publisher={Elsevier},
  url ={https://doi.org/10.1016/0024-3795(81)90026-4}
}

@article{kasum1981chemical,
  title={Chemical graph theory. III. On the permanental polynomial},
  author={Kasum, D and Trinajsti{\'c}, N and Gutman, I},
  journal={Croatica Chemica Acta},
  volume={54},
  number={3},
  pages={321--328},
  year={1981},
  publisher={Hrvatsko kemijsko dru{\v{s}}tvo},
  url = {https://hrcak.srce.hr/194320}
}

@article{cash2000permanental,
  title={The permanental polynomial},
  author={Cash, Gordon G},
  journal={Journal of Chemical Information and Computer Sciences},
  volume={40},
  number={5},
  pages={1203--1206},
  year={2000},
  publisher={ACS Publications},
  url={https://doi.org/10.1021/ci000031d}
}

@article{gutman1998permanents,
  title={Permanents of adjacency matrices and their dependence on molecular structure},
  author={Gutman, Ivan},
  journal={Polycyclic Aromatic Compounds},
  volume={12},
  number={4},
  pages={281--287},
  year={1998},
  publisher={Taylor \& Francis},
  url={https://doi.org/10.1080/10406639808233845}
}

@article{liu2013characterizing,
  title={On the characterizing properties of the permanental polynomials of graphs},
  author={Liu, Shunyi and Zhang, Heping},
  journal={Linear Algebra and its Applications},
  volume={438},
  number={1},
  pages={157--172},
  year={2013},
  publisher={Elsevier},
  url ={http://dx.doi.org/10.1016/j.laa.2012.08.026}
}

@article{zhang2015per,
  title={Per-spectral characterizations of some edge-deleted subgraphs of a complete graph},
  author={Zhang, Heping and Wu, Tingzeng and Lai, Hong-Jian},
  journal={Linear and Multilinear Algebra},
  volume={63},
  number={2},
  pages={397--410},
  year={2015},
  publisher={Taylor \& Francis},
  url={http://dx.doi.org/10.1080/03081087.2013.869592}
}

@article{wu2015per,
  title={Per-spectral characterizations of graphs with extremal per-nullity},
  author={Wu, Tingzeng and Zhang, Heping},
  journal={Linear Algebra and its Applications},
  volume={484},
  pages={13--26},
  year={2015},
  publisher={Elsevier},
  url = {http://dx.doi.org/10.1016/j.laa.2015.06.018}
}

@article{yu1999permanent,
  title={The permanent rank of a matrix},
  author={Yu, Yang},
  journal={J. Comb. Theory, Ser. A},
  volume={85},
  number={2},
  pages={237--242},
  year={1999},
  url={https://doi.org/10.1006/jcta.1998.2904}
}

@article{alon1989nowhere,
  author       = {Noga Alon and
                  Michael Tarsi},
  title        = {A nowhere-zero point in liner mappings},
  journal      = {Combinatorica},
  volume       = {9},
  number       = {4},
  pages        = {393--396},
  year         = {1989},
  url          = {https://doi.org/10.1007/BF02125351},
}

@article{fanai2012permanent,
  title={Permanent rank and transversals},
  author={Fana{\i}, Hamid-Reza},
  journal={Australasian Journal of Combinatorics},
  volume={53},
  pages={285--288},
  year={2012}
}

@book{minc1984permanents, place={Cambridge}, series={Encyclopedia of Mathematics and its Applications}, title={Permanents}, publisher={Cambridge University Press}, author={Minc, Henryk}, year={1984}, collection={Encyclopedia of Mathematics and its Applications}}

@article{mccuaig2004polya,
  title={P{\'o}lya's permanent problem},
  author={McCuaig, William},
  journal={The Electronic Journal of Combinatorics},
  pages={R79},
  year={2004},
  url = {https://doi.org/10.37236/1832}
}

@article{polya1913aufgabe,
  author    = {P{\'o}lya, George},
  title     = {Aufgabe 424},
  journal   = {Archiv der Mathematik und Physik},
  volume    = {20},
  pages     = {271},
  year      = {1913}
}

@article{beasley1969maximal,
  title={Maximal Groups on Which the Permanent is Multiplicative},
  author={LeRoy B. Beasley},
  journal={Canadian Journal of Mathematics},
  year={1969},
  volume={21},
  pages={495 - 497},
  url={https://doi.org/10.4153/CJM-1969-055-2}
}

@article{carlen2006inequality,
  title={An Inequality of Hadamard Type for Permanents},
  author={Eric A. Carlen and Elliott H. Lieb and Michael Loss},
  journal={Methods and applications of analysis},
  year={2005},
  volume={13},
  pages={1-18},
  url={https://doi.org/10.4310/maa.2006.v13.n1.a1}
}

@article{heuvers1988characterization,
title = {A characterization of the permanent function by the Binet-Cauchy theorem},
journal = {Linear Algebra and its Applications},
volume = {101},
pages = {49-72},
year = {1988},
issn = {0024-3795},
url = {https://doi.org/10.1016/0024-3795(88)90142-5},
author = {Konrad J. Heuvers and L.J. Cummings and K.P.S. {Bhaskara Rao}}
}

@Inbook{lieb2002proofs,
author="Lieb, Elliott H.",
title="Proofs of some Conjectures on Permanents",
bookTitle="Inequalities: Selecta of Elliott H. Lieb",
year="2002",
publisher="Springer Berlin Heidelberg",
address="Berlin, Heidelberg",
pages="101--108",
isbn="978-3-642-55925-9",
url="https://doi.org/10.1007/978-3-642-55925-9_11"
}

@article{marcus1965permanents,
author = {Marvin Marcus and Henryk Minc},
title = {Permanents},
journal = {The American Mathematical Monthly},
volume = {72},
number = {6},
pages = {577--591},
year = {1965},
publisher = {Taylor \& Francis},
URL = {https://doi.org/10.1080/00029890.1965.11970575}
}

@article{Marcus1963ThePA,
  title={The permanent analogue of the Hadamard determinant theorem},
  author={Marvin Marcus},
  journal={Bulletin of the American Mathematical Society},
  year={1963},
  volume={69},
  pages={494-496},
  url={https://doi.org/10.1090/S0002-9904-1963-10975-1}
}

@article{harary1953notion,
  title={On the notion of balance of a signed graph.},
  author={Frank Harary},
  journal={Michigan Mathematical Journal},
  year={1953},
  volume={2},
  pages={143-146},
  url={https://doi.org/10.1307/MMJ%2F1028989917}
}

@article{HARARY1980131,
  title = {A simple algorithm to detect balance in signed graphs},
  author = {Harary, Frank and Kabell, Jerald A.},
  journal = {Mathematical Social Sciences},
  volume = {1},
  number = {1},
  pages = {131--136},
  year = {1980},
url={https://doi.org/10.1016/0165-4896(80)90010-4}
}

@article{zaslavsky1982signed,
title = {Signed graphs},
journal = {Discrete Applied Mathematics},
volume = {4},
number = {1},
pages = {47-74},
year = {1982},
issn = {0166-218X},
url = {https://doi.org/10.1016/0166-218X(82)90033-6},
author = {Thomas Zaslavsky}
}

@ARTICLE{zaslavsky2013matrices,
       author = {{Zaslavsky}, Thomas},
        title = "{Matrices in the Theory of Signed Simple Graphs}",
      journal = {arXiv e-prints},
         year = 2013,
        month = mar,
          eid = {arXiv:1303.3083},
        pages = {arXiv:1303.3083},
          url = {https://doi.org/10.48550/arXiv.1303.3083},
}

@article{tang2022permanental,
author = {Tang, Zikai and Li, Qiyue and Deng, Hanyuan},
year = {2022},
month = {02},
pages = {14–20},
title = {On the permanental polynomial and permanental sum of signed graphs},
volume = {10},
journal = {Discrete Mathematics Letters},
url = {https://doi.org/10.47443/dml.2022.005}
}

@article{singh2024note,
  title={A note on graphs with purely imaginary per-spectrum},
  author={Singh, Ranveer and Wankhede, Hitesh},
  journal={Applied Mathematics and Computation},
  volume={475},
  pages={128754},
  year={2024},
  publisher={Elsevier},
url ={https://doi.org/10.48550/arXiv.2211.13072}
}

@article{bapat2024computing,
  title={Computing the permanental polynomial of 4k-intercyclic bipartite graphs},
  author={Ravindra B. Bapat and Ranveer Singh and Hitesh Wankhede},
  journal={ArXiv},
  year={2024},
  volume={abs/2411.14238},
  url={https://doi.org/10.48550/arXiv.2411.14238}
}

@book{BrualdiCvetkovic2008,
  author    = {Richard A. Brualdi and Drago{\v{s}} Cvetkovi{\'c}},
  title     = {A Combinatorial Approach to Matrix Theory and Its Applications},
  publisher = {Chapman \& Hall/CRC},
  year      = {2008},
    url = {https://doi.org/10.1201/9781420082241}
}

@article{hopcroft1973n,
  author  = {Hopcroft, John E. and Karp, Richard M.},
  title   = {An $n^{5/2}$ Algorithm for Maximum Matchings in Bipartite Graphs},
  journal = {SIAM Journal on Computing},
  volume  = {2},
  number  = {4},
  pages   = {225--231},
  year    = {1973},
url = {https://doi.org/10.1109/SWAT.1971.1}
}

\end{document}